\RequirePackage{fix-cm}
\documentclass{svjour3}                     
\smartqed  

\usepackage{graphicx}
\usepackage{geometry}
\usepackage{amsmath}               
\usepackage{amsfonts}              
\usepackage{leftidx}
\usepackage{tikz-cd}
\usepackage{mathtools}
\usepackage[all]{xy}
\usepackage[T1]{fontenc}
\usepackage{float}

\newcommand{\bigslant}[2]{{\raisebox{.2em}{$#1$}\left/\raisebox{-.2em}{$#2$}\right.}}
\newcommand{\adhoc}[2]{{\raisebox{-.2em}{$#1$}\oplus\raisebox{.2em}{$#2$}}}

\begin{document}

\title{Examples of foliations with infinite dimensional special cohomology
}


\author{Andrzej Czarnecki         \and
        Pawe\l{} Ra\'{z}ny 
}


\institute{Andrzej Czarnecki \at Jagiellonian University, \L{}ojasiewicza 6, 30-348 Krakow, Poland\\\email{andrzejczarnecki01@gmail.com} \and Pawe\l{} Ra\'{z}ny \at Jagiellonian University, \L{}ojasiewicza 6, 30-348 Krakow, Poland\\\email{pawel.razny@student.uj.edu.pl}}


\maketitle

\begin{abstract}
We present examples of foliations with infinite dimensional basic symplectic and complex cohomologies, along with a general sufficient condition for such phenomena. This puts restrictions on possible generalizations of several finiteness results from Riemannian foliations to any broader class. The examples are also noteworthy for the unusual behaviour of their basic de Rham cohomology.
\end{abstract}

\section{Introduction.}

The initial purpose of this short paper was to present exmples of foliations with infinite dimensional special basic cohomologies. A simple algebraic Lemma \ref{niesk} (based on methods used in \cite{Ang}) enables us to relate special basic cohomologies to ordinary basic cohomology of a foliation. We thus obtain interesting examples of infinite dimensional basic Bott-Chern and Aeppli cohomologies (in the transversely holomorphic case), and basic $(d+d^{\Lambda})$- and $dd^{\Lambda}$-cohomologies (in the transversely symplectic case). These cohomologies (and especially their non-foliated counterparts) are subject of extensive study (cf. \cite{A1}, \cite{appli}, \cite{Ang}, \cite{Ida}, \cite{My}, \cite{Yau}), in particular it can be proved (cf. \cite{My}) that for Riemannian foliations they all are finite dimensional. Our examples amount to say that certain compactness conditions in those proofs cannot be dropped, which is by no means obvious (cf. \cite{Ma}, which works for both compact and non-compact manifolds). The examples also violate various dualities present in the Riemannian case. We will begin the next section with a review of all relevant notions.

However, the behaviour of special basic cohomologies in our examples revealed an interesting picture. Of all the peculiarities collected in the last section, we single out the following: the basic cohomology of our transversely holomorphic example has infinite dimension in degree 2 and 4, and finite dimension in degree 3. To the extent of our knowledge, no example of basic cohomology reverting to finite dimension was ever given. It is worth pointing out that we know essentially only two ways of producing infinite dimensional basic cohomology: either Schwarz' \cite{schwarz} or Ghys' \cite{ghys}. We feel it is important to better understand how the infinite dimensional basic cohomology may arise in non-Riemannian foliations. Lemma \ref{niesk} and its application to the transversely holomorphic example shed some light on this matter.

\section{Foliations.}

\subsection{Transverse structures.}

We start with a brief review of some basic facts about foliations and transverse structures. The interested reader is referred to \cite{M1} for a thorough exposition. All manifolds are assumed to be compact.

\begin{definition}
A codimension $n$ foliation $\mathcal{F}$ on a smooth manifold $M$ is given by the following data:

\begin{itemize}
\item an open cover $\mathcal{U}:=\{U_i\}_{i\in I}$ of M;
\item a $n$-dimensional smooth manifold $T_0$;
\item for each $U_i\in\mathcal{U}$ a submersion $f_i\,:\, U_i\longrightarrow T_0$ with connected fibres (called \emph{plaques});
\item for each intersections $U_i\cap U_j\neq\emptyset$ a local diffeomorphism $\gamma_{ij}$ of $T_0$ such that $f_j=\gamma_{ij}\circ f_i$.
\end{itemize}

The last condition ensures that the plaques glue nicely to form a partition of $M$ by submanifolds of codimension $n$, called \emph{leaves} of $\mathcal{F}$.
\end{definition}

We call $T=\coprod\limits_{U_i\in\mathcal{U}}f_i(U_i)$ the transverse manifold of $\mathcal{F}$. The local diffeomorphisms $\gamma_{ij}$ generate a pseudogroup $\Gamma$ of transformations on $T$ (called the \emph{holonomy pseudogroup}). The space of leaves $\bigslant{M}{\mathcal{F}}$ of the foliation $\mathcal{F}$ can be identified with $\bigslant{T}{\Gamma}$. We note that neither $T$ nor $\bigslant{T}{\Gamma}$ need to be compact, even if $M$ is.

\begin{definition}
A smooth form $\omega$ on $M$ is called basic if for any vector field $X$ tangent to the leaves of $\mathcal{F}$ we have

$$
i_X\omega= i_Xd\omega= 0
$$

Basic forms are in one-to-one correspondence with $\Gamma$-invariant smooth forms on $T$, a point of view that we will take below.
\end{definition}

It is clear that $d\omega$ is basic for any basic form $\omega$. Hence the set of basic forms of $\mathcal{F}$, $\Omega^{\bullet}(M\slash\mathcal{F})$, is a subcomplex of the de Rham complex of $M$. We define the \emph{basic cohomology} (or sometimes \emph{basic de Rham cohomology} if other basic cohomologies will be in play) of $\mathcal{F}$ to be the cohomology of this subcomplex and denote it by $H^{\bullet}(M\slash\mathcal{F})$.

A transverse structure to $\mathcal{F}$ is any $\Gamma$-invariant structure on $T$. We will need the following examples.

\begin{definition}
$\mathcal{F}$ is said to be \emph{transversely orientable} if $T$ is orientable and all the $\gamma_{ij}$ are orientation preserving. $\mathcal{F}$ is said to be \emph{homologically orientable} if the top basic cohomology space $H^{n}(M\slash\mathcal{F})=\mathbb{R}$. Contrary to the non-foliated case, these two notions are not equivalent as we will see later on.
\end{definition}

\begin{definition}
$\mathcal{F}$ is said to be \emph{Riemannian} if $T$ has a $\Gamma$-invariant Riemannian metric. 
\end{definition}

We emphasize here a strong property of the Riemannian foliations: their basic cohomology is always finite dimensional. We refer the reader to \cite{E1} for the proof applicable also to the other cohomologies considered below.

From now on, we will restrict our attention to foliations of even codimension $2n$.

\begin{definition}
$\mathcal{F}$ is said to be \emph{transversely symplectic} if $T$ admits a $\Gamma$-invariant closed non-degenerate 2-form $\omega$, called a transverse symplectic form. Again, contrary to the non-foliated case this does not imply homological orientability, although it makes $\mathcal{F}$ transversally oriented. 
\end{definition}

\begin{definition}
$\mathcal{F}$ is said to be \emph{transversely holomorphic} if $T$ admits a complex structure that makes all the $\gamma_{ij}$ holomorphic.
\end{definition}

\begin{definition}
A foliation is said to be \emph{Hermitian} if it is both transversely holomorphic and Riemannian. We emphasize that neither of these condition implies the other.
\end{definition}

If $\mathcal{F}$ is transversely holomorphic, we have the standard decomposition of the space of complex valued forms $\Omega^{\bullet}({M\slash\mathcal{F},\mathbb{C}})$ into forms of type $(p,q)$ and $d$ decomposes as $\partial+\bar{\partial}$ of orders (1,0) and (0,1), respectively. We can then define the basic Dolbeault double complex $\left(\Omega^{\bullet,\bullet}({M\slash\mathcal{F},\mathbb{C}}),\partial,\bar{\partial}\right)$, basic Dolbeault cohomology $H^{\bullet}_{\bar{\partial}}(M\slash\mathcal{F})=\bigslant{\text{ker}\bar{\partial}}{\text{im }\bar{\partial}}$, and basic Fr\"{o}licher spectral sequence just like in the non-foliated case. We note that for a Hermitian foliation the basic Dolbeault cohomology is again finite dimensional and even without this assumption the basic Fr\"{o}licher sequence converges to the basic cohomology.

\subsection{Bott-Chern and Aeppli cohomology theories.}

Let $M$ be a manifold endowed with a transversely holomorphic foliation $\mathcal{F}$ of complex codimension $n$. Using the operators $\partial$ and $\bar{partial}$ above we can construct the \emph{basic Bott-Chern} and \emph{basic Aeppli} cohomologies of $\mathcal{F}$:

\begin{align*}
H^{\bullet,\bullet}_{BC}(M\slash\mathcal{F}):=&\frac{\text{ker} \partial \cap \text{ker} \bar{\partial}\cap \Omega^{\bullet,\bullet}(M\slash\mathcal{F})}{\text{im } \partial\bar{\partial} \cap \Omega^{\bullet,\bullet}(M\slash\mathcal{F})}  \\
H^{\bullet,\bullet}_{A}(M\slash\mathcal{F}):=&\frac{\text{ker} \partial\bar{\partial} \cap \Omega^{\bullet,\bullet}(M\slash\mathcal{F})}{\text{im } \partial \cap \text{im } \bar{\partial}\cap \Omega^{\bullet,\bullet}(M\slash\mathcal{F}))}
\end{align*}

Assuming the foliation is Hermitian, we have some restrictions on these cohomologies. We recall the results from \cite{My}.

\begin{theorem}
If $M$ is a compact manifold endowed with a Hermitian foliation $\mathcal{F}$, then the dimensions of $H^{\bullet,\bullet}_{BC}(M\slash\mathcal{F})$ and $H^{\bullet,\bullet}_A(M\slash\mathcal{F})$ are finite.
\end{theorem}

\begin{corollary} \label{dual}
If $M$ is a compact manifold endowed with a Hermitian homologically orientable foliation $\mathcal{F}$, then there is an isomorphism

$$
H^{p,q}_{BC}(M\slash\mathcal{F}) \longrightarrow H^{n-p,n-q}_A(M\slash\mathcal{F})
$$

induced by the transverse Hodge star operator.
 
\end{corollary}

We also note that for a K\"{a}hler manifold these cohomologies are isomorphic to Dolbeault cohomology -- so their failure to do so measures how far a manifold is from being K\"{a}hler. The same applies to the foliated case. 

\subsection{$dd^{\Lambda}$- and ($d+d^\Lambda$)-cohomology theories.}

Let $\mathcal{F}$ be a transversely symplectic foliation of codimension $2n$ on $M$ with a basic symplectic form $\omega$. Let us start by defining the symplectic star operator for $\mathcal{F}$. The transverse symplectic form defines a non-degenerate pairing $\tilde{G}$ of the vector fields on the transverse manifold. We can then extend it to a (non-degenerate) pairing $G$ on basic forms.

\begin{definition}
The symplectic star operator is a linear operator

$$
*_s\,:\,\Omega^{k}(M\slash\mathcal{F})\longrightarrow\Omega^{2n-k}(M\slash\mathcal{F})
$$

uniquely defined by the formula

$$
\alpha_1\wedge *_s\alpha_2 = G(\alpha_1,\alpha_2)\frac{\omega^n}{n!}
$$

where $\alpha_1$ and $\alpha_2$ are arbitrary basic $k$-forms.
\end{definition}

The symplectic star operator is an isomorphism, cf. \cite{Ma}, and \cite{BC} for an account in the foliated case. Out of many operators connected to the symplectic star, we only use 

$$
d^{\Lambda}\alpha:=(-1)^{k+1}*_s d*_s(\alpha)
$$

where $\alpha$ is again a basic $k$-form. We point out the relation $dd^{\Lambda}+d^{\Lambda}d=0$. We can use this operator to define basic cohomology theories similar to those reviewed in the previous subsection:

\begin{align*}
H^{\bullet}_{d^{\Lambda}}(M\slash\mathcal{F})&:=\frac{Ker(d^{\Lambda})}{Im(d^{\Lambda})}
\\ H^{\bullet}_{d+d^{\Lambda}}(M\slash\mathcal{F})&:=\frac{Ker(d+d^{\Lambda})}{Im(dd^{\Lambda})}
\\ H^{\bullet}_{dd^{\Lambda}}(M\slash\mathcal{F})&:=\frac{Ker(dd^{\Lambda})}{Im(d)+Im(d^{\Lambda})}
\end{align*}

It is easy to see that the basic $d^{\Lambda}$-cohomology is simply the basic cohomology with reversed gradation and hence will not concern us. In a manner slightly different from Corollary \ref{dual}, the symplectic star gives Poincar\'e dualities in the last two cohomologies

\begin{align*}
H^{p}_{dd^{\Lambda}}(M\slash\mathcal{F})&\longrightarrow H^{2n-p}_{dd^{\Lambda}}(M\slash\mathcal{F}) \\
H^{p}_{d+d^{\Lambda}}(M\slash\mathcal{F})&\longrightarrow H^{2n-p}_{d+d^{\Lambda}}(M\slash\mathcal{F})
\end{align*}

We note that this does not depend on any transverse Riemannian structure (which may well not exist), nor on the homological orientation.

It can be shown, using \cite{My} and \cite{Yau}, that for a Riemannian transversely symplectic foliation $H^{\bullet}_{dd^{\Lambda}}$ and $H^{\bullet}_{d+d^{\Lambda}}$ are finite dimensional.

\section{An algebraic lemma and its consequences.}

Let $I_1$, $I_2$, $I_{12}$, $K_1$, $K_2$, $K_{12}$ be vector spaces satisfying

$$
\begin{array}{ccccccc}
I_{12}\subset I_1,I_2 &;& I_1\subset K_1 &;& I_2\subset K_2 &;& K_1,K_2\subset K_{12}
\end{array}
$$

Then the following lemma holds:

\begin{lemma}\label{niesk}
If $\bigslant{K_1}{I_1}$ or $\bigslant{K_2}{I_2}$ have infinite dimension then $\bigslant{\left(K_1\cap K_2\right)}{I_{12}}$ or $\bigslant{K_{12}}{\left(I_1+I_2\right)}$ has infinite dimension as well.
\end{lemma}

\begin{proof}
Without loss of generality let us assume that $\bigslant{K_1}{I_1}$ is infinite dimensional. Then there are two sequences

$$
\begin{tikzcd}
\bigslant{\left(K_1\cap I_2\right)}{I_{12}} \arrow{r}{f'} & \bigslant{K_1}{I_1} \arrow{r}{g'} & \bigslant{K_{12}}{\left(I_1+I_2\right)}\\
\bigslant{\left(K_1\cap K_2\right)}{I_{12}} \arrow{r}{f"} & \bigslant{K_1}{I_1} \arrow{r}{g"} & \bigslant{K_{12}}{\left(I_1+K_2\right)}
\end{tikzcd}
$$


It is easy to see that these sequences are exact in the middle (since the appropriate kernel and image are classes represented by elements of $\left(K_1\cap I_2\right)$). If both $\bigslant{\left(K_1\cap K_2\right)}{I_{12}}$ and $\bigslant{K_{12}}{\left(I_1+I_2\right)}$ have finite dimension, then so do $\bigslant{\left(K_1\cap I_2\right)}{I_{12}}$ and $\bigslant{K_{12}}{\left(I_1+K_2\right)}$ since they are smaller. But then the middle term has finite dimension by exactness, a contradiction.
\end{proof}

We can apply this lemma to transversely symplectic and transversely holomorphic structures to get

\begin{proposition} \label{sympnie}
If $\mathcal{F}$ is a transversely symplectic foliation for which $H^k(M\slash\mathcal{F})$ is infinitely dimensional then $H^k_{d+d^{\Lambda}}(M\slash\mathcal{F})$ or $H^k_{dd^{\Lambda}}(M\slash\mathcal{F})$ have infinite dimension.
\end{proposition}

\begin{proof}
In the lemma take:

$$
\begin{array}{ccccc}
K_1= \text{ker } d && K_2=\text{ker } d^{\Lambda} && K_{12}=\text{ker } dd^{\Lambda} \\
I_1=\text{im } d && I_2=\text{im } d^{\Lambda} && I_{12}=\text{im } dd^{\Lambda}
\end{array}
$$

\end{proof}

\begin{proposition}\label{holonie}
If $\mathcal{F}$ is a transversely holomorphic foliation for which $H^{p,q}_{\bar{\partial}}(M\slash\mathcal{F})$ is infinite dimensional then $H^{p,q}_{BC}(M\slash\mathcal{F})$ or $H^{p,q}_{A}(M\slash\mathcal{F})$ have infinite dimension.
\end{proposition}

\begin{proof}
In the lemma take:

$$
\begin{array}{ccccc}
K_1=\text{ker } \partial && K_2=\text{ker } \bar{\partial} && K_{12}=\text{ker } \partial\bar{\partial} \\
I_1=\text{im } \partial && I_2=\text{im } \bar{\partial} && I_{12}=\text{im } \partial\bar{\partial}
\end{array}
$$

\end{proof}

\begin{corollary}\label{holonie2}
If $\mathcal{F}$ is a transversely holomorphic foliation for which $H^{k}(M\slash\mathcal{F})$ is infinitely dimensional then, for some $(p,q)$ satisfying $p+q=k$, $H^{p,q}_{BC}(M\slash\mathcal{F})$ or $H^{p,q}_{A}(M\slash\mathcal{F})$ has infinite dimension.
\end{corollary}

\begin{proof}
By the previous proposition, it is sufficient to prove that Dolbeault cohomology has infinite dimension for some $(p,q)$ with $p+q=k$. This is obvious since the basic Fr\"{o}licher spectral sequence converges to the basic cohomology of $\mathcal{F}$ and so the dimensions of the entries on the first page must be greater than of those in the limit.
\end{proof}

\section{Transversely symplectic example.}

Consider, as in \cite{he}, a map of the 2-torus $\mathbb{T}^2$ given by the matrix $A=\left[\begin{smallmatrix}1 & 1\\ 0 & 1\end{smallmatrix}\right]$. We form a suspension of this map, $\left(M,\mathcal{F}_A\right)$: a codimension two foliation on $\bigslant{\mathbb{T}^2\times [0,1]}{(t,0)\sim (At,1)}$. The plaques of this foliation are the lines $[0,1]\times \{t_0\}$ and $\mathbb{T}^2$ can be taken for the transverse manifold, with the pseudgroup $\Gamma$ the infinite cyclic group generated by $A$. Since $\det A=1$, this foliation is transversely symplectic with the standard symplectic form $dx\wedge dy$ on $\mathbb{T}^2$.

We determine the basic complex. Any basic function $f$ must satisfy $f(x,y)=f(x+y,y)$. Taking an irrational $y_0$ and any $x_0$ we see that $f$ does not depend on the first coordinate, since it is constant on $\{(x_0+ny_0,y_0)\}$, dense in $\{(x,y_0)\}$. Therefore the basic functions correspond to smooth functions on a circle

\begin{align*}
\Omega^0(M\slash\mathcal{F}_A)&=\{f(y)\,|\,f\in\mathcal{C}^{\infty}(\mathbb{S}^1)\}
\end{align*}

In a similar fashion we see that

\begin{align*}
\Omega^1(M\slash\mathcal{F}_A)&=\{f(y)dy\,|\,f\in\mathcal{C}^{\infty}(\mathbb{S}^1)\} \\
\Omega^2(M\slash\mathcal{F}_A)&=\{f(y)dx\wedge dy\,|\,f\in\mathcal{C}^{\infty}(\mathbb{S}^1)\}
\end{align*}

It is then easy to see that the basic cohomology is

\begin{enumerate}
\item $H^0(M\slash\mathcal{F}_A)=H^1(M\slash\mathcal{F}_A)=\mathbb{R}$
\item $H^2(M\slash\mathcal{F}_A)=\mathcal{C}^{\infty}(\mathbb{S}^1)$
\end{enumerate}

We note again that this precludes this foliation from being Riemannian.

We compute $H^{\bullet}_{d+d^{\Lambda}}(M\slash\mathcal{F}_A)$ and $H^{\bullet}_{dd^{\Lambda}}(M\slash\mathcal{F}_A)$. Observe that $dd^\Lambda=-d^{\Lambda} d=0$. In degree 2 and 0, it factors through the trivial spaces $\Omega^3(M\slash\mathcal{F}_A)$ and 
$\Omega^{-1}(M\slash\mathcal{F}_A)$, respectively. In degree 1 $dd^\Lambda f(y)dy = -d^{\Lambda}d f(y)dy=0$, or because $*_s=id$ on $\Omega^1(M\slash\mathcal{F}_A)$. Consequently

\begin{enumerate}
\item $H^{0}_{d+d^{\Lambda}}(M\slash\mathcal{F}_A)=H^{2}_{d+d^{\Lambda}}(M\slash\mathcal{F}_A)=\mathbb{R}$ 
\item $H^{1}_{d+d^{\Lambda}}(M\slash\mathcal{F}_A)=\mathcal{C}^{\infty}(\mathbb{S}^1)$
\end{enumerate}

\begin{enumerate}
\item $H^{0}_{dd^{\Lambda}}(M\slash\mathcal{F}_A)=H^{2}_{dd^{\Lambda}}(M\slash\mathcal{F}_A)=\mathcal{C}^{\infty}(\mathbb{S}^1)$
\item $H^{1}_{dd^{\Lambda}}(M\slash\mathcal{F}_A)=\mathbb{R}$
\end{enumerate}  

\section{Transversely holomorphic example.}

To provide a transversely holomorphic foliation exhibiting the similar behaviour we mimic the construction presented above. We take the map of the 4-torus $\mathbb{T}^4$ induced by the matrix $A=\left[\begin{smallmatrix}1 & 0 & 1 & 0\\ 0 & 1 & 0 & 1 \\ 0 & 0 & 1 & 0\\ 0 & 0 & 0 & 1 \end{smallmatrix}\right]$. We form a suspension of this map, $\left(M,\mathcal{F}_A\right)$: a codimension four foliation on $\bigslant{\mathbb{T}^4\times [0,1]}{(t,0)\sim (At,1)}$. Since $A$ is in $Gl(2,\mathbb{C})\subset Gl(4,\mathbb{R})$ this foliation is transversally holomorphic with the complex structure induced from $\mathbb{T}^4$.

On the transverse manifold $\mathbb{T}^4$ we will use real coordinates $(x_1,y_1,x_2,y_2)$ (better suited for the suspension) and then switch to complex coordinates $(w,z)=(x_1+iy_1,x_2+iy_2)$ (better suited for the bigradation of the complex forms). We will describe the $A$-invariant forms, computing only the 2-forms explicitly as an example. The operators $\partial$, $\bar{\partial}$ and $\partial\bar{\partial}$ will prove to be not too complicated and we will proceed to compute basic de Rham, Dolbeault, Aepli and Bott-Chern cohomologies.

\subsection{Invariant forms.}

As in the previous example we can easily see the invariant complex functions to depend only on the last two real coordinates, or on the complex coordinate $z$. Hence $\Omega^0(M\slash\mathcal{F}_A,\mathbb{C})=\mathcal{C}^{\infty}(\mathbb{T}^2,\mathbb{C})$.

An $A$-invariant complex 2-form on the 4-torus is a skew-symmetric matrix

$$
\alpha=\left[\begin{smallmatrix}0 & f_1 & f_2 & f_3 \\
-f_1 & 0 & f_4 & f_5 \\
-f_2 & -f_4 & 0 & f_6 \\
-f_3 & -f_5 & -f_6 & 0\end{smallmatrix}\right]
$$

where each entry is an invariant function, satisfying

$$
\alpha = A^t \alpha A
$$

that amounts to $f_1=0$ and $f_3=f_4$. This gives 

\begin{align*}
\Omega^2(M\slash\mathcal{F}_A,\mathbb{C})= &\{f_2(x_2,y_2)dx_1\wedge dx_2\} \\
&\oplus {}{} \{f_3(x_2,y_2)\left(dx_1\wedge dy_2+dy_1\wedge dx_2\right)\} \\
&\oplus {}{} \{f_5(x_2,y_2)dx_2\wedge dy_2\} \\
&\oplus {}{} \{f_6(x_2,y_2)dx_2\wedge dy_2\}
\end{align*} 

which we will now rewrite in complex coordinates

\begin{align*}
\Omega^2(M\slash\mathcal{F}_A,\mathbb{C})= &\{b(z)dw\wedge dz\} \\
&\oplus {}{} \{c(z)\left(dw\wedge d\bar{z}+d\bar{w}\wedge dz\right)\} \\
&\oplus {}{} \{e(z)d\bar{w}\wedge d\bar{z}\} \\
&\oplus {}{} \{f(z)dz\wedge d\bar{z}\}
\end{align*} 

Note that any complex function of the complex coordinate is to be smooth, not holomorphic. We present all the invariant forms with the complex bigradation $\Omega^{\bullet,\bullet}(M\slash\mathcal{F}_A,\mathbb{C})$, indicating where the differentials are obviously trivial. We use a generic letter $g$ for functions in degrees other than 2, since the labeling will play no role there.

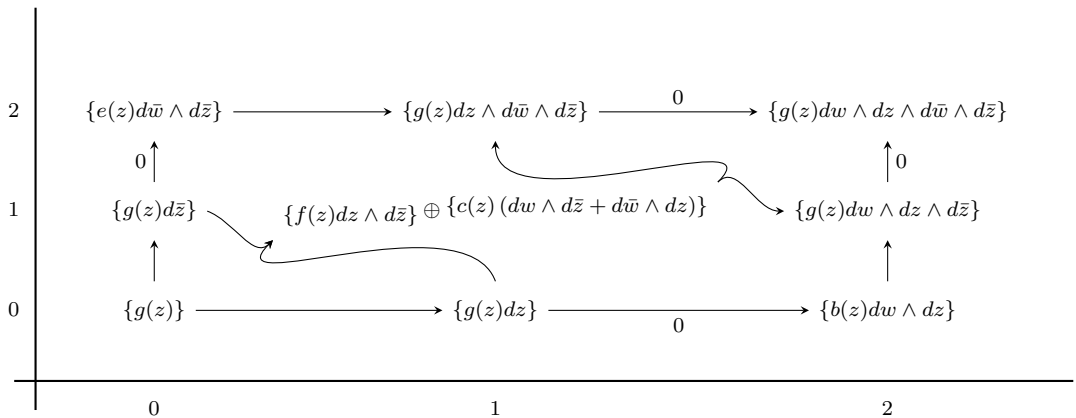
\begin{figure}[H]\label{f1}
\begin{tikzpicture}
\matrix (m) [matrix of math nodes, nodes in empty cells,nodes={minimum width=10pt, minimum height=20pt,outer sep=1pt}, row sep=2em, column sep=2em]{
\quad\strut &      &     &     & \\
2 & \left\{e(z)d\bar{w}\wedge d\bar{z}\right\} & \left\{g(z)dz\wedge d\bar{w}\wedge d\bar{z}\right\} & \left\{g(z)dw\wedge dz\wedge d\bar{w}\wedge d\bar{z}\right\} \\
1 & \left\{g(z)d\bar{z}\right\} & \adhoc{\left\{f(z)dz\wedge d\bar{z}\right\}}{\left\{c(z)\left(dw\wedge d\bar{z}+d\bar{w}\wedge dz\right)\right\}} & \left\{g(z)dw\wedge dz\wedge d\bar{z}\right\} \\
0 & \left\{g(z)\right\} & \left\{g(z)dz\right\} & \left\{b(z)dw\wedge dz\right\} \\
\quad\strut &   0  &  1  &  2  & \strut \\};
\draw[-stealth] (m-3-2.east) to [out=-20,in=225] (m-3-3.south west);
\draw[-stealth] (m-4-3.north) to [out=110,in=225] (m-3-3.south west);
\draw[-stealth] (m-3-3.north east) to [out=45,in=270] (m-2-3.south);
\draw[-stealth] (m-3-3.north east) to [out=45,in=180] (m-3-4.west);

\draw[-stealth] (m-3-2.north) to node[left] {0} (m-2-2.south);
\draw[-stealth] (m-3-4.north) to node[right] {0} (m-2-4.south);

\draw[-stealth] (m-4-2.north) to (m-3-2.south);
\draw[-stealth] (m-4-4.north) to (m-3-4.south);

\draw[-stealth] (m-4-3.east) to node[below] {0} (m-4-4.west);
\draw[-stealth] (m-2-3.east) to node[above] {0} (m-2-4.west);

\draw[-stealth] (m-2-2.east) to (m-2-3.west);
\draw[-stealth] (m-4-2.east) to (m-4-3.west);

\draw[thick] (m-1-1.east) -- (m-5-1.east) ;
\draw[thick] (m-5-1.north) -- (m-5-5.north) ;
\end{tikzpicture}
\end{figure}

\newpage

The curvy arrows are meant to indicate that $d\Omega^1(M\slash\mathcal{F}_A)$ is contained in the $\{f(z)dz\wedge d\bar{z}\}$ term of $\Omega^{1,1}(M\slash\mathcal{F}_A)$ and $d\Omega^{1,1}(M\slash\mathcal{F}_A)=d\{c(z)\left(dw\wedge d\bar{z}+d\bar{w}\wedge dz\right)\}$. Note that the diagram shows that $\partial\bar{\partial}$ can be non-zero only on the 0-forms.

\subsection{De Rham cohomology.} We compute the basic cohomology over $\mathbb{C}$. Some of the spaces involved can be described in terms of cohomology of the complex torus $\mathbb{T}^2$ -- parts of the diagram above clearly repeat parts of $\Omega^{\bullet,\bullet}(\mathbb{T}^2,\mathbb{C})$ -- which is not complicated, because the torus is K\"{a}hler.

Since the differentials are quite simple too, we hope that the reader will have no trouble justifying the claims below.

\begin{enumerate}
\item $H^0(M\slash\mathcal{F}_A)\simeq H^0(\mathbb{T}^2)=\mathbb{C}$;
\item $H^1(M\slash\mathcal{F}_A)\simeq H^1(\mathbb{T}^2)=\mathbb{C}^2$;
\item $H^2(M\slash\mathcal{F}_A)\simeq V\oplus H^2(\mathbb{T}^2)=V\oplus\mathbb{C}$, where $V$ is an infinite dimensional space $\{\bar{\partial} b-\partial c=\bar{\partial} c-\partial e=0\}$ easily seen to be infinite dimensional; none of these closed forms is exact since the image $d\Omega^1(M\slash\mathcal{F}_A)\subset \{f(z)dz\wedge d\bar{z}\}$; the term $H^2(\mathbb{T}^2)$ follows from this inclusion;
\item $H^3(M\slash\mathcal{F}_A)\simeq \left(H^2(\mathbb{T}^2)\right)^2=\mathbb{C}^2$ since the $dw$ and $d\bar{w}$ factors do not interfere in any way;
\item $H^4=\mathcal{C}^{\infty}(\mathbb{T}^2,\mathbb{C})$ since the image $d\Omega^3(M\slash\mathcal{F}_A)$ is trivial;
\end{enumerate}

We present the three complex cohomologies in diagrams explaining their entries below each one.

\subsection{Dolbeault cohomology} 

\begin{figure}[H]
\begin{tikzpicture}
\matrix (m) [matrix of math nodes, nodes in empty cells,nodes={minimum width=10pt, minimum height=10pt,outer sep=5pt}, row sep=2em, column sep=2em]{
\quad\strut &      \makebox[50pt]{}&     \makebox[50pt]{}&     \makebox[50pt]{}& \\
2 \strut&   \mathcal{C}^{\infty}(\mathbb{T}^2,\mathbb{C})  &  \mathbb{C}   &   \mathcal{C}^{\infty}(\mathbb{T}^2,\mathbb{C})  & \strut\\
1 \strut&  \mathbb{C}    &  \mathbb{C}\oplus\mathbb{C} &  \mathbb{C}   & \strut\\
0 \strut&  \mathbb{C}    &  \mathbb{C}   &  \mathbb{C}   & \strut\\
\quad\strut & \makebox[50pt]{0}&  \makebox[50pt]{1}&  \makebox[50pt]{2}& \strut \\};

\draw[line width=0.25mm] (m-1-1.east) -- (m-5-1.east) ;
\draw[line width=0.25mm] (m-5-1.north) -- (m-5-5.north) ;

\draw[dashed,line width=0.1mm] (m-1-2.east) -- (m-5-2.east) ;
\draw[dashed,line width=0.1mm] (m-1-3.east) -- (m-5-3.east) ;

\draw[dashed,line width=0.1mm] (m-4-1.north) -- (m-4-5.north) ;
\draw[dashed,line width=0.1mm] (m-3-1.north) -- (m-3-5.north) ;

\end{tikzpicture}
\end{figure}

\begin{enumerate}
\item $H^{0,0}_{\bar{\partial}}$ is represented by constant functions;
\item $H^{1,0}_{\bar{\partial}}$ and $H^{2,0}_{{\bar{\partial}}}$ are represented by antiholomorphic functions;
\item $H^{1,1}_{{\bar{\partial}}}$ splits as $H^{1,1}_{{\bar{\partial}}}(\mathbb{T}^2)$ and antiholomorphic functions;
\item $H^{0,1}_{{\bar{\partial}}}$ is again represented by antiholomorphic functions -- note that $\bar{\partial}\{g(z)\}$ are precisely all functions divided by antiholomorphic ones; the same reasoning applies to $H^{2,1}_{{\bar{\partial}}}$ and $H^{1,2}_{{\bar{\partial}}}$;
\item the remaining spaces are $\mathcal{C}^{\infty}(\mathbb{T}^2,\mathbb{C})$ since the relevant differentials are trivial;
\end{enumerate}

\subsection{Bott-Chern cohomology}

\begin{figure}[H]
\begin{tikzpicture}
\matrix (m) [matrix of math nodes, nodes in empty cells,nodes={minimum width=10pt, minimum height=10pt,outer sep=5pt}, row sep=2em, column sep=2em]{
\quad\strut &      \makebox[50pt]{}&     \makebox[50pt]{}&     \makebox[50pt]{}& \\
2 \strut&   \mathbb{C}  &  \mathcal{C}^{\infty}(\mathbb{T}^2,\mathbb{C})   &   \mathcal{C}^{\infty}(\mathbb{T}^2,\mathbb{C})  & \strut\\
1 \strut&  \mathbb{C}    &  \mathbb{C}\oplus \mathbb{C}  &  \mathcal{C}^{\infty}(\mathbb{T}^2,\mathbb{C})   & \strut\\
0 \strut&  \mathbb{C}    &  \mathbb{C}   &  \mathbb{C}   & \strut\\
\quad\strut & \makebox[50pt]{0}&  \makebox[50pt]{1}&  \makebox[50pt]{2}& \strut \\};

\draw[line width=0.25mm] (m-1-1.east) -- (m-5-1.east) ;
\draw[line width=0.25mm] (m-5-1.north) -- (m-5-5.north) ;

\draw[dashed,line width=0.1mm] (m-1-2.east) -- (m-5-2.east) ;
\draw[dashed,line width=0.1mm] (m-1-3.east) -- (m-5-3.east) ;

\draw[dashed,line width=0.1mm] (m-4-1.north) -- (m-4-5.north) ;
\draw[dashed,line width=0.1mm] (m-3-1.north) -- (m-3-5.north) ;

\end{tikzpicture}
\end{figure}

\begin{enumerate}
\item $H^{0,0}_{BC}$ is represented by constant functions;
\item $H^{1,0}_{BC}$ and $H^{1,0}_{BC}$ are represented by holomorphic and antiholomorphic functions, respectively;
\item so are $H^{2,0}_{BC}$ and $H^{2,0}_{BC}$;
\item $H^{1,1}_{BC}$ is $H^{1,1}_{BC}(\mathbb{T}^2)$ plus constant functions coming from the second summand in $\Omega^{1,1}(M\slash\mathcal{F}_A)$;
\item $H^{2,1}_{BC}=H^{1,2}_{BC}=H^{2,2}_{BC}=\mathcal{C}^{\infty}(\mathbb{T}^2,\mathbb{C})$ since each of the $\partial$, $\bar{\partial}$, and $\partial\bar{\partial}$ is trivial in these cases;
\end{enumerate}

\subsection{Aeppli cohomology}

\begin{figure}[H]
\begin{tikzpicture}
\matrix (m) [matrix of math nodes, nodes in empty cells,nodes={minimum width=10pt, minimum height=20pt,outer sep=5pt}, row sep=2em, column sep=2em]{
\quad\strut &      \makebox[50pt]{}&     \makebox[50pt]{}&     \makebox[50pt]{}& \\
2 \strut&   \mathcal{C}^{\infty}(\mathbb{T}^2,\mathbb{C})  &  \mathbb{C}   &   \mathcal{C}^{\infty}(\mathbb{T}^2,\mathbb{C})  & \strut\\
1 \strut&  \mathbb{C}    &  \mathbb{C}\oplus\mathcal{C}^{\infty}(\mathbb{T}^2,\mathbb{C}) &  \mathbb{C}    & \strut\\
0 \strut&  \mathbb{C}    &  \mathbb{C}   &   \mathcal{C}^{\infty}(\mathbb{T}^2,\mathbb{C}) & \strut\\
\quad\strut & \makebox[50pt]{0}&  \makebox[50pt]{1}&  \makebox[50pt]{2}& \strut \\};

\draw[line width=0.25mm] (m-1-1.east) -- (m-5-1.east) ;
\draw[line width=0.25mm] (m-5-1.north) -- (m-5-5.north) ;

\draw[dashed,line width=0.1mm] (m-1-2.east) -- (m-5-2.east) ;
\draw[dashed,line width=0.1mm] (m-1-3.east) -- (m-5-3.east) ;

\draw[dashed,line width=0.1mm] (m-4-1.north) -- (m-4-5.north) ;
\draw[dashed,line width=0.1mm] (m-3-1.north) -- (m-3-5.north) ;

\end{tikzpicture}
\end{figure}

\begin{enumerate}
\item $H^{0,0}_{A}$ is represented by constant functions;
\item $H^{1,0}_{A}=H^{1,0}_{A}(\mathbb{T}^2)$ and $H^{0,1}_{A}=H^{0,1}_{A}(\mathbb{T}^2)$;
\item $H^{2,0}_{A}$ and $H^{0,2}_{A}$ are all the relevant forms, since the differentials are all zero in these cases; 
\item $H^{1,1}_{A}$ splits into the $H^{1,1}_{BC}(\mathbb{T}^2)=H^2(\mathbb{T}^2)$ and $\mathcal{C}^{\infty}(\mathbb{T}^2,\mathbb{C})$ since the $\partial\bar{\partial}$ is trivial;
\item $H^{2,1}_{A}=H^{1,2}_{A}$ are again $H^{1,2}_{A}(\mathbb{T}^2)=H^{2,1}_{A}(\mathbb{T}^2)$;
\item $H^{2,2}_{A}=\mathcal{C}^{\infty}(\mathbb{T}^2,\mathbb{C})$ since all the differentials are trivial;
\end{enumerate}

\section{Conclusions}

We close this paper summarising some interesting properties of the given examples.

\begin{remark}
The transversely symplectic example highlights that the infinite dimension of symplectic cohomology in dimension $k$ may stem, by Proposition \ref{sympnie}, form infinite dimension of de Rham basic cohomology in degree $k$ ($H^{2}_{dd^{\Lambda}}(M\slash\mathcal{F}_A)$), or degree $2n-k$, via Poincar\'e duality ($H^{0}_{dd^{\Lambda}}(M\slash\mathcal{F}_A)$), or indeed it can be unprovoked by any of these ($H^{1}_{d+d^{\Lambda}}(M\slash\mathcal{F}_A)$).

\end{remark}

\begin{remark}
The transversely holomorphic example exhibits infinite dimensional basic, Dolbeault, Bott-Chern and Aeppli cohomologies. The basic Aeppli cohomology is infinite dimensional in bidegrees $(2,0)$, $(1,1)$, $(0,2)$, and $(2,2)$. In bidegree $(1,1)$ both the basic Dolbeault cohomology and its adjoint counterpart (the basic $\partial$-cohomology) are finite dimensional. This shows that also the basic Aeppli cohomology can be infinite dimensional without the help of basic Dolbeault cohomology and Proposition \ref{sympnie}. The same thing happens for the basic Bott-Chern cohomology in bidegrees $(2,1)$ and $(1,2)$. We also note that while $H^{1,1}_{A}(M\slash\mathcal{F}_A)$ could be perhaps explained by infinite dimension of $H^2(M\slash\mathcal{F}_A)$, it is not the case for either $H^{2,1}_{BC}(M\slash\mathcal{F}_A)$ or $H^{1,2}_{BC}(M\slash\mathcal{F}_A)$.
\end{remark}

\begin{remark}
The example proves that Poincar\'e duality between Bott-Chern and Aeppli cohomology may fail in absence of Riemannian metric. We also point out that for the same reason the Dolbeault cohomology of the example does not exibit Serre duality.
\end{remark}

\begin{remark}
It is also worth pointing out that the basic de Rham cohomology of this example is infinite dimensional in degrees 2 and 4, but reverts to being finite dimensional in degree 3. To the extent of our knowledge such example have not been described before.
\end{remark}

\begin{remark}
The property of reverting to finite dimension is important for the further developements in this field. The richest geometry in the transversely symplectic setting is the K\"{a}hler structure, and short of that -- the hard Lefshetz property, that


$$
\begin{tikzcd}
\Omega^{n-k}(M\slash\mathcal{F}_A) \arrow{rr}{\wedge\omega^k} && \Omega^{n+k}(M\slash\mathcal{F}_A)
\end{tikzcd}
$$

is an epimorphism. It is a theorem that this property forces the map

$$
H^{\bullet}_{d+d^{\Lambda}}(M\slash\mathcal{F}_A)\ni [\alpha] \mapsto [\alpha] \in H^{\bullet}(M\slash\mathcal{F}_A)
$$

to be epimorphic as well, cf. \cite{BC,Yau}. It is natural to ask for examples where infinite dimension of the former is derived from infinite dimension of the latter via the hard Lefschetz property. However, as we pointed out in the introduction, $H^0(M\slash\mathcal{F}_A)$ and $H^1(M\slash\mathcal{F}_A)$ are always finite dimensional and there is no known example of infinite dimensional basic cohomology without the infinite dimension in the top degree.
\end{remark}

\end{document}